\documentclass[12pt]{amsart}
\usepackage{amssymb}
\usepackage{amscd}
\usepackage{amsmath}
\usepackage{amsfonts}
\usepackage{latexsym}
\usepackage{hyperref}
\usepackage[all]{xy}
\usepackage{verbatim}
\usepackage[margin=1in]{geometry}

\newtheorem{theorem}{Theorem}[section]
\newtheorem{lemma}[theorem]{Lemma}
\newtheorem{proposition}[theorem]{Proposition}
\newtheorem{corollary}[theorem]{Corollary} 
\theoremstyle{definition}  
\newtheorem{definition}[theorem]{Definition}
\newtheorem{example}[theorem]{Example}
\newtheorem{conjecture}[theorem]{Conjecture}  

\newtheorem{remark}[theorem]{Remark}

\newcommand{\id}{\text{id}}

\newcommand{\Fun}{\text{Fun}}
\newcommand{\FPdim}{\text{FPdim}} 

\renewcommand{\Vec}{\operatorname{\operatorname{\mathsf{Vec}}}}

\DeclareMathOperator{\Aut}{\operatorname{\mathsf{Aut}}}
\DeclareMathOperator{\End}{\operatorname{\mathsf{End}}}

\DeclareMathOperator{\Rep}{\operatorname{\mathsf{Rep}}}
\DeclareMathOperator{\Hom}{\operatorname{\mathsf{Hom}}}

\newcommand{\rev}{\text{rev}}
\newcommand{\intg}{\text{int}}

\newcommand{\B}{\mathcal{B}}
\newcommand{\C}{\mathcal{C}}
\newcommand{\D}{\mathcal{D}}
\newcommand{\E}{\mathcal{E}}

\newcommand{\Z}{\mathcal{Z}}

\newcommand{\M}{\mathcal{M}}
\newcommand{\A}{\mathcal{A}}

\newcommand{\TY}{\mathcal{T}\mathcal{Y}}

\renewcommand{\O}{\mathcal{O}}

\newcommand{\bt}{\boxtimes}
\newcommand{\ot}{\otimes}

\newcommand{\beq}{\begin{equation}}
\newcommand{\eeq}{\end{equation}}

\newcommand{\bpf}{\begin{proof}}
\newcommand{\epf}{\end{proof}}

\newcommand{\bth}{\begin{theorem}}
\renewcommand{\eth}{\end{theorem}}
\newcommand{\bpr}{\begin{proposition}}
\newcommand{\epr}{\end{proposition}}
\newcommand{\ble}{\begin{lemma}}
\newcommand{\ele}{\end{lemma}}
\newcommand{\bco}{\begin{corollary}}
\newcommand{\eco}{\end{corollary}}
\newcommand{\bde}{\begin{definition}}
\newcommand{\ede}{\end{definition}}
\newcommand{\bex}{\begin{example}}
\newcommand{\eex}{\end{example}}
\newcommand{\bre}{\begin{remark}}
\newcommand{\ere}{\end{remark}}
\newcommand{\bcj}{\begin{conjecture}}
\newcommand{\ecj}{\end{conjecture}}

\begin{document}
\title[Braid group  representations]
{On the braid group  representations  coming from weakly group-theoretical fusion categories}
\author{Jason Green}
\address{Department of Mathematics and Statistics,
University of New Hampshire,  Durham, NH 03824, USA}
\email{jag1032@wildcats.unh.edu}
\author{Dmitri Nikshych}
\address{Department of Mathematics and Statistics,
University of New Hampshire,  Durham, NH 03824, USA}
\email{dmitri.nikshych@unh.edu}

\dedicatory{Dedicated to  Nicol\'as Andruskiewitsch on the occasion of his 60th birthday}

\begin{abstract}
We prove that representations of the braid groups coming from weakly group-theoretical braided
fusion categories have finite images.
\end{abstract}  

\maketitle
\baselineskip=18pt


\section{Introduction}

From  an object $X$ of a braided fusion category  one gets representations of the braid group $B_n$
by automorphisms of the tensor power $X^{\ot n},\ n\geq 1$. A natural problem is to investigate  the images  of  these representations. 
A conjecture of Naidu and Rowell \cite{NR} states that these images  are finite for all $n\geq 1$
if and only if the Frobenius-Perron dimension of $X$ is a square root of an integer.
In  particular, this means that the braid group images are finite for all objects of braided fusion categories
of integral Frobenius-Perron dimension. 

This conjecture is still open, but a number of partial results is known.  It was proved in \cite{ERW} that it is true
for group-theoretical braided fusion categories  (these are precisely fusion subcategories
of the representation categories of twisted group doubles).
In \cite{RW} this conjecture was verified for categories $SO(N)_2$. The latter are examples of  {\em
weakly group-theoretical} braided fusion categories. This class of  fusion categories
was introduced in \cite{ENO3}, see Definition~\ref{wgt def}.  Unlike  the group-theoretical ones, these categories
can contain objects of non-integral Frobenius-Perron dimension.
Braided weakly group-theoretical fusion categories
can be obtained from the category of vector spaces by a sequence of extensions
and equivariantizations by finite groups. Thus, in principle, they can be described in terms of group cohomology.
All known examples of fusion categories of integral dimension (in particular, all representation categories of
known semisimple quasi-Hopf algebras) are weakly group-theoretical. 

In this paper, we extend previously known results by proving that the images of the braid group representations associated 
with objects of a weakly group-theoretical braided fusion category $\B$ are finite (Theorem~\ref{main thm}). The idea of the proof is
to describe restrictions of these representations to certain finite index subgroups of $B_n$ in terms of representations 
coming from a group-theoretical category. Note that $\B$, in general, cannot be embedded into a group-theoretical fusion category. 

The contents of the paper are as follows. 

In Section \ref{sect weakly GT} we recall basic definitions and facts
related to weakly group-theoretical fusion categories \cite{ENO3}.    

In Section \ref{sect G-crossed} we recall  $G$-crossed 
braided fusion categories \cite{T} and their equivariantizations.  
We explain in Corollary~\ref{VecGw Gcrossunity} 
that the structure morphisms of the braided $G$-crossed fusion category $\Vec_G^\alpha$ can be chosen so that its values on simple objects
are roots of unity.  
A particularly useful for us result is 
Proposition~\ref{Natale's} (which is a consequence of \cite{Na})  stating that the center of a  weakly group-theoretical fusion 
category is equivalent to an equivariantization of a braided $G$-crossed fusion category with a pointed trivial component. 

In Section~\ref{sect det} we discuss determinants of certain canonical automorphisms in graded fusion categories. 
We show  in Proposition~\ref{alpha C invariant} that
a certain power of the determinant of the associator between regular homogeneous objects in an integral $G$-graded fusion category $\C$
determines a cohomology class $\alpha^\C \in H^3(G,\, k^\times)$ that is an invariant of $\C$.    As a consequence,
determinants of the crossed braidings between the homogeneous regular objects of $G$ are roots of $1$ (Corollary~\ref{blocks are roots of 1}).
We use these determinants as a technical tool in this paper, but they  seem to be interesting in their own right.

Braid groups and their representations coming from braided fusion categories are recalled in Section~\ref{sect braid}.

The fiber product of braided $G$-crossed fusion categories is discussed in Section~\ref{sect fiber}. 
Proposition~\ref{when Ce pointed} gives a sufficient condition for the ``fiber double" of a $G$-crossed
braided fusion category to be group-theoretical.

Section~\ref{sect finiteness} contains the proof of our main result.  

Finally, in Section~\ref{sect TY} we explicitly describe  the braid group representations coming from the centers
of Tambara-Yamagami categories \cite{GNN, TY}.

{\bf Acknowledgements.} 
We thank Pavel Etingof for many valuable comments on the preliminary version of this work that allowed us, in particular, 
to simplify the proof of the main result.
We are also grateful to  Michael M\"uger, Deepak Naidu, Eric Rowell, and Zhenghan Wang for helpful discussions. 
The work of the second author was supported  by  the  National  Science  Foundation  under  
grant DMS-1801198.


\section{Weakly integral and weakly group-theoretical fusion categories}
\label{sect weakly GT}

Throughout this article we work over an algebraically closed field $k$ of characteristic $0$.  

We refer the reader to \cite{EGNO} for the basic facts about fusion categories. Here  we recall
some terminology and constructions that will be used in this paper.

For a fusion category $\C$ let $\O(\C)$ denote the set of isomorphism classes of simple objects of $\C$.
By a {\em fusion subcategory} of $\C$ we will mean a full tensor subcategory of $\C$.   The smallest
fusion subcategory of $\C$ is generated by the unit object of $\C$ and is equivalent to $\Vec$, the fusion category
of finite-dimensional $k$-vector spaces.

Objects of $\C$   of integer  Frobenius-Perron dimension span a fusion subcategory of $\C$ which we will denote $\C_{\intg}$. 
We say that $\C$ is {\em integral} if $\C_{\intg}=\C$, i.e., if $\FPdim(X) \in \mathbb{Z}$ for all objects $X$ in $\C$.
A fusion category is integral if and only if it is equivalent to the representation category of a semisimple
quasi-Hopf algebra. We say that $\C$ is {\em weakly integral}  \cite{ENO3} if $\FPdim(\C) \in \mathbb{Z}$. 
An example of a weakly integral fusion category that is not integral is an Ising category \cite[Appendix B]{DGNO}.

A {\em grading} of a fusion category $\C$  by a group $G$ is a map $\deg: \O(\C)\to G$ such that for all $X,Y,Z\in \O(\C)$
one has $\deg(X)\deg(Y) =\deg(Z)$ when $Z$ is contained in $X\ot Y$.  In this case we have a decomposition
\begin{equation}
\label{grading}
\C =\bigoplus_{g\in G}\, \C_g,
\end{equation}
where $\C_g$ is the full additive subcategory of $\C$ generated by all objects of degree $g\in G$.  The subcategory $\C_e$
corresponding to the identity element $e\in G$ is a fusion subcategory of $\C$ and is called the {\em trivial component}
of the grading.  A grading is {\em faithful} if the map $\deg: \O(\C)\to G$  is surjective. In this case we say that $\C$
is a {\em $G$-extension} (or, simply, an {\em extension}) of $\C_e$. 

The simplest  example of a  graded fusion category is a {\em pointed} fusion category, i.e., a category $\C$ in which every simple 
object is invertible.  Indeed, in  this case $G=\O(\C)$ has a structure of a group and $\C$ is a $G$-extension of $\Vec$.
Any pointed fusion category  is equivalent to the category $\Vec_G^\alpha$ 
of finite-dimensional $G$-graded vector spaces (for some finite group $G$) with 
the associativity constraint given by a  $3$-cocycle $\alpha \in Z^3(G,\,k^\times)$. 

A fusion category is {\em nilpotent} \cite{GN} if it can be obtained from $\Vec$ by a sequence of extensions. 

\begin{proposition}
\label{GN grading} 
Let $\C$ be a  weakly integral fusion category.  
\begin{enumerate}
\item[(i)] $\FPdim(X)^2 \in \mathbb{Z}$ for any $X \in \O(\C)$. 
\item[(ii)] The map $\deg :\O (\C )\to\mathbb{Q}^{\times}_{>0}/(\mathbb{Q}^{\times}_{>0})^2$
that takes $X\in\O (\C )$ to the image of $\FPdim (X)^2$
in $\mathbb{Q}^{\times}_{>0}/(\mathbb{Q}^{\times}_{>0})^2$ is
a grading of $\C$. 
\end{enumerate}
\end{proposition}
\begin{proof}
(i) was proved in \cite[Proposition 8.27]{ENO1}  and (ii) was proved in \cite[Theorem 3.10]{GN}.
\end{proof}

\begin{corollary}
\label{dimension grading}
Let $\C$ be a  weakly integral fusion category.  There is a canonical faithful grading
of $\C$ by an elementary Abelian $2$-group $G(\C)$,
\[
\C =\bigoplus_{g\in G(\C)}\, \C_g,
\]
such that $\C_e = \C_{\intg}$ and there are square free integers $N_g$ such that   $\FPdim(X_g) \in \mathbb{Z}\sqrt{N_g}$
for all $X_g \in \C_g,\, g\in G(\C)$. 
\end{corollary}

Recall that a fusion category $\C$ is called {\em group-theoretical} \cite{ENO1} if $\C$ is categorically Morita equivalent 
to a pointed fusion category.  It is known that $\C$ is group-theoretical if and only if its center $\Z(\C)$ is braided equivalent
to the representation category of a twisted group double $\Rep(D^\alpha(G))$ for some finite group $G$ and a $3$-cocycle
$\alpha \in Z^3(G,\,k^\times)$. 

\begin{definition}
\label{wgt def}
A fusion category is {\em weakly group-theoretical} \cite{ENO3} if it is categorically Morita 
equivalent  to a nilpotent fusion category. 
\end{definition}

A weakly group-theoretical fusion category $\C$ is weakly integral.
All known examples of weakly integral fusion categories are weakly group-theoretical. In particular, it was shown in \cite{ENO3} that
a fusion category $\C$ such that $\FPdim(\C)=p^aq^b$, where $p$ and $q$ are primes, 
is weakly group-theoretical.


\section{Braided $G$-crossed fusion categories}
\label{sect G-crossed}

Let $G$ be a finite group. The following notion is due to Turaev.

\begin{definition}
A {\em braided  $G$-crossed fusion category} is a
fusion category $\C$ equipped with  the following structures:
\begin{enumerate}
\item[(i)] an action of $G$ on $\C$, i.e., a collection of    
tensor autoequivalences $g$  of $\C$  along with natural isomorphisms 
\begin{equation}
\label{mu and gamma}
\mu_g(X,Y): g(X)\ot g(Y) \xrightarrow{\sim} g(X\ot Y) \quad \text{and}  \quad \gamma_{g,h}(X): g\circ h(X) \xrightarrow{\sim} gh(X)
\end{equation}
for all  $X,Y\in \C,\, g,h\in G,$ satisfying monoidal functor structure axioms;
\item[(ii)] a (not necessarily faithful) grading $\C=\bigoplus_{g\in G}\C_g$;  
\item[(iii)]  a natural isomorphism 
\begin{equation}
\label{crossed braiding} 
c_{X,Y}: X\ot Y\simeq g(Y)\ot X, \qquad
X\in \C_g,\, Y \in \C,
\end{equation} 
called a {\em $G$-crossed braiding}.
\end{enumerate}
These data should satisfy the following conditions:
\begin{enumerate}
\item[(a)] the diagram
\begin{equation}
\label{G-hex0}
\xymatrix @C=0.6in @R=0.45in{
g(X) \ot g(Y) \ar[rr]^{c_{g(X), g(Y)}} & &
{ghg^{-1}}\circ g(Y) \ot g(X)  \ar[d]^{\gamma_{ghg^{-1},g}(Y) \ot \id_{g(X)}} & \\
g(X\ot Y) \ar[u]^{\mu_g(X,Y)^{-1}}
\ar[d]_{g(c_{X,Y})} &  &  {gh}(Y) \ot g(X) \\
g(h(Y) \ot X) \ar[rr]^{\mu_g(h(Y),X)^{-1}} & &
g\circ h(Y) \ot g(X) \ar[u]_{\gamma_{g,h}(Y) \ot \id_{g(X)}},
}
\end{equation}
commutes for all $g,h \in G$ and objects $X\in \C_h,\, Y\in \C$,
\item[(b)] the diagram
\begin{equation}
\label{G-hex1}
\xymatrix{
& (X\ot Y)\ot Z \ar[dl]_{\alpha_{X,Y,Z}} \ar[dr]^{c_{X,Y}\ot \id_Z} & \\
X\ot (Y\ot Z) \ar[d]_{c_{X,Y\ot Z}} &  & (g(Y)\ot X) \ot Z \ar[d]^{\alpha_{g(Y),X,Z}} \\
g(Y \ot Z) \ot X \ar[d]_{\mu_g(Y,Z)^{-1}\ot \id_X} & & g(Y) \ot (X \ot Z) \ar[d]^{\id_{g(Y)}\ot c_{X,Z}} \\
(g(Y) \ot g(Z)) \ot X  \ar[rr]^{\alpha_{g(Y),g(Z), X}} & & g(Y) \ot (g(Z) \ot X)
}
\end{equation}
commutes for all $g \in G$ and objects $X \in \C_g, Y,Z \in \C,$ and
\item[(c)] the diagram
\begin{equation}
\label{G-hex2}
\xymatrix{
& X\ot (Y\ot Z) \ar[dr]^{\id_X \ot c_{Y,Z}} & \\
(X\ot Y)\ot Z  \ar[ur]^{\alpha_{X,Y,Z}} &  & X\ot (h(Z)\ot Y) \ar[d]^{\alpha^{-1}_{X, h(Z),Y}} \\
{gh}(Z) \ot (X \ot Y)  \ar[u]^{c_{X\ot Y, Z}^{-1}}  & & (X\ot h(Z))\ot Y \ar[d]^{c_{X, h(Z)}\ot \id_Y} \\
g\circ h(Z) \ot (X  \ot Y)  \ar[u]^{\gamma_{g,h}(Z)\ot \id_{X\ot Y}}  \ar[rr]^{\alpha^{-1}_{g\circ h(Z),X, Y}} & & (g\circ h(Z) \ot X)  \ot Y.
}
\end{equation}
commutes for all $g,h \in G$ and objects $X\in \C_g,\,Y\in \C_h, Z \in \C$.
\end{enumerate}
Here $\alpha$ denotes the associativity constraint of $\C$.
\end{definition}

\begin{remark}
The trivial component $\C_e$ of a braided $G$-crossed fusion category $\C$ 
is a braided fusion category and $G$ acts on it by braided autoequivalences. 
\end{remark}

\begin{definition}
We say that a braided $G$-crossed fusion category $\C$ is {\em non-degenerate}  if its grading is faithful 
and $\C_e$ is a non-degenerate braided fusion category.
\end{definition}

Note that $\C$ is a non-degenerate braided $G$-crossed fusion category if and only if $\C^G$ is a non-degenerate
braided fusion category.

\begin{definition}
Let $\C$ and $\C'$ be braided $G$-crossed fusion categories. A {\em braided $G$-crossed tensor functor}
$F: \C \to \C'$  is  a tensor functor preserving the $G$-grading along with a natural isomorphism
of tensor functors 
\begin{equation}
\label{eta g}
\eta_g : F  \circ g \to g  \circ  F,\qquad g\in G,
\end{equation}
such that the  diagrams 
 \begin{equation}
 \label{Gfun1}
 \xymatrix{
 F\circ g\circ h(X) \ar[rrr]^{F(\gamma_{g,h})} \ar[d] _{\eta_g(h(X))} &&&  F\circ {gh}(X)   \ar[dd]^{\eta_{gh}(X)} \\
 g\circ F\circ h(X) \ar[d]_{g(\eta_h(X))} &&& \\
 g\circ h\circ F(X) \ar[rrr]^{\gamma'_{g,h}(F(X))} &&& {gh}\circ F(X),
 }
 \end{equation}
 and
  \begin{equation}
  \label{Gfun2}
 \xymatrix{
 F \circ g(X) \ot F \circ g(Y) \ar[rr]^{\varphi_{g(X),\, g(Y)}} \ar[dd]_{\eta_g(X)\ot \eta_g(Y)} &&
 F(g(X)\ot g(Y)) \ar[rr]^{F((\mu_g)_{X,Y})}&&  F \circ g(X\ot  Y)  \ar[dd]^{\eta_g(X \ot Y)} \\
 \\
 g \circ F(X) \ot g \circ F(Y)  \ar[rr]^{(\mu'_g)_{F(X), F(Y)}} && g(F(X)\ot F(Y)) \ar[rr]^{g(\varphi_{X,Y})} && g \circ F(X \ot Y),
 }
 \end{equation}
   commute for all $g, h\in G$  and $X,Y\in \C$ and the diagram
  \begin{equation}
   \label{Gfun3}
 \xymatrix{
 F(X) \ot F(Y)  \ar[rrr]^{c'_{F(X), F(Y)}} \ar[dd]_{\varphi_{X,Y}} &&& g \circ F(Y) \ot F(X) \ar[d]^{\eta_g(Y)^{-1}\ot \text{id}_{F(X)} }\\
 &&& F  \circ g(Y) \ot F(X) \ar[d]^{\varphi_{g(Y),X}} \\
 F(X \ot Y) \ar[rrr]^{F(c_{X,Y})}  &&& F(g(Y) \ot X)
 }
 \end{equation}
 commutes for all $X \in \C_g,\,g\in G,$  and $Y\in \C$. 
 Here $\varphi_{X,Y}: F(X)\ot F(Y) \xrightarrow{\sim} F(X\ot Y)$ denotes the tensor functor structure of $F$
 and $\gamma,\, \mu,\, c$ (respectively, $\gamma',\, \mu',\, c'$) denote the structure isomorphisms 
 of $\C$ (respectively, $\C'$).  
\end{definition}

\begin{example}
\label{VecGw example}
Let $G$ be a finite group and $\alpha\in Z^3(G,\, k^\times)$ be a $3$-cocycle.  There is a canonical 
braided $G$-crossed category structure on $\Vec_G^\alpha$ defined as follows. 
We view $\Vec_G^\alpha$ as a semisimple Abelian $k$-linear category with simple objects denoted by $x\,(x\in G)$.
The associativity constraint is given by 
\begin{equation}
\label{alpha}
\alpha_{x,y,z}= \alpha(x,y,z) \id_{xyz}: (x\ot y)\ot z \to x \ot (y \ot z).
\end{equation}
The action of $g\in G$ is by $g(x)=gxg^{-1}$ with the tensor functor structure of $g$ given by
\begin{equation}
\label{mu}
\mu_g(y,z) = \frac{\alpha(gyg^{-1}, gzg^{-1}, g)\alpha(g,y,z)}{\alpha(gyg^{-1}, g, z)} \,\id_{gyzg^{-1}}: 
g(y)\ot g(z)\to g(y\ot z),
\end{equation}
the monoidal functor structure on the functor $G\to \Aut(\Vec)$ given by
\begin{equation}
\label{gamma}
\gamma_{g,h}(x) = \frac{\alpha(g, hxh^{-1}, h)}{\alpha(g,h,x)\alpha(ghxh^{-1}g^{-1}, g, h)} \,\id_{ghxh^{-1}g^{-1}}:
g(h(x))\to (gh)(x),
\end{equation}
and the crossed braiding given by
\begin{equation}
\label{c}
c_{g,x} = \id_{gx} : g\ot x \to g(x) \ot g, 
\end{equation}
for all $x,y,z, g,h\in G$. 

That the above maps $\mu$ and $\gamma$ determine an action of $G$ on $\Vec_G^\alpha$
follows from the identities
\begin{eqnarray}
\frac{\mu_g(xy,z)\mu_g(x,y)}{\mu_g(x,yz)\mu_g(y,z)} &=& \frac{\alpha(gxg^{-1}, gyg^{-1}, gzg^{-1})}{\alpha(x,y,z)}, \\
\gamma_{gh, f}(x) \gamma_{g,h}(fxf^{-1})  &=& \gamma_{g,hf}(x) \gamma_{h,f}(x), \\
\frac{\mu_{gh}(x,y)}{\mu_h(x,y)\mu_g(hxh^{-1}, hyh^{-1})} &=&  \frac{\gamma_{g,h}(xy)}{\gamma_{g,h}(x)\gamma_{g,h}(y)},
\end{eqnarray}
for all $x,y,z,f,g,h\in G$. 
The diagram \eqref{G-hex0} commutes thanks to the identity
\begin{equation}
\frac{\mu_g(x,y)}{\mu_g(xyx^{-1},\, x)} = \frac{\gamma_{gxg^{-1}, g}(y)}{\gamma_{g,x}(y)}, 
\end{equation}
for all $g,x,y\in G$,
while the diagrams \eqref{G-hex1} and \eqref{G-hex2} are the definitions of $\mu$ and $\gamma$. 
The above identities are consequences the $3$-cocycle equation for $\alpha$ (and are well known). 
\end{example}

\begin{proposition}
\label{VecGw Gcrossequiv}
The braided $G$-crossed category structure $\mu,\, \gamma,\, c$  
on $\Vec_G^\alpha$ defined by formulas \eqref{mu}, \eqref{gamma}, and \eqref{c} in Example~\ref{VecGw example}
is unique up to a braided $G$-crossed  equivalence. 
\end{proposition}
\begin{proof}
If $\mu',\, \gamma',\, c'$ is another braided $G$-crossed category structure on $\Vec_G^\alpha$ 
then the identity tensor functor $\id_{\Vec_G^\alpha}$ (with the tensor structure  $\varphi_{x,y}=\id_{xy}$) 
equipped with the natural isomorphism 
\begin{equation}
\label{etagx}
\eta_g(x)= \frac{c'_{g,x}}{c_{g,x}}\, \id_{gxg^{-1}} : g(x)\to g(x)
\end{equation}
establishes an equivalence between these braided $G$-crossed categories.  Indeed, comparing diagrams
\eqref{G-hex1} and \eqref{G-hex2} for them we get
\[
\frac{\gamma'_{g,h}(x)}{\gamma_{g,h}(x)} = \frac{\eta_{gh}(x)}{\eta_g(hxh^{-1})\eta_h(x)}
\quad \text{and} \quad
\frac{\mu'_g(x,y)}{\mu_g(x,y)} = \frac{\eta_g(xy)}{\eta_g(x)\eta_g(y)}
\]
for all $x,y,g,h\in G$, which gives commutativity of \eqref{Gfun1} and \eqref{Gfun2}.  The commutativity of 
\eqref{Gfun3} is immediate from the definition \eqref{etagx} of $\eta$. 
\end{proof}

The following result is well known. We include its proof for the sake of completeness.

\begin{lemma}
\label{rootof1lemma}
Let $G$ be a finite group.  Let $\alpha \in  Z^{n}(G,\,k^\times)$ be an $n$-cocycle, $n\geq 1$.
Then there exists an $n$-cocycle $\alpha'$  cohomologous to $\alpha$  such that  the values of $\alpha'$
are $|G|$th roots of unity.
\end{lemma}
\begin{proof}
Let  $g_0,g_1, \ldots, g_n \in G$. The $n$-cocycle condition for $\alpha$ is
\[
\alpha(g_1, \ldots, g_n)\alpha(g_0g_1, \ldots, g_n)^{-1}
\cdots \alpha(g_0, \ldots, g_{n-1}g_n)^{(-1)^n}\alpha(g_0, \ldots, g_n)^{(-1)^{n+1}}
= 1.
\]
Taking a product over $g_0\in G$ we get
\begin{equation}
\label{deepak's}
\prod_{g_0 \in G}\alpha(g_1, \ldots, g_n)\alpha(g_0g_1, \ldots, g_n)^{-1}
\cdots \alpha(g_0, \ldots, g_{n-1}g_n)^{(-1)^n}\alpha(g_0, \ldots, g_n)^{(-1)^{n+1}}
= 1.
\end{equation}

Let $m=|G|$. Choose a function  $r: G^{n-1} \to k^\times$ such that 
\[
r(g_1, \ldots, g_{n-1})^{-m} =
\displaystyle \prod_{g_0 \in G} \alpha(g_0, g_1, \ldots, g_{n-1}).
\]
Then \eqref{deepak's} can be rewritten as
\[
\alpha(g_1, \ldots, g_n)^m 
r(g_2, \cdots, g_n)^{m}
\cdots r(g_1, \ldots, g_{n-1}g_n)^{(-1)^{n} m} r(g_1, \ldots, g_{n-1})^{(-1)^{n+1} m} =1.
\]
Take $\alpha' := \alpha d(r)$, then $\alpha'(g_1, \ldots, g_n)^m=1$.
%
\end{proof}

\begin{corollary}
\label{VecGw Gcrossunity}
The braided $G$-crossed fusion category $\Vec_G^\alpha$ is equivalent to one in which all scalars
\eqref{alpha}, \eqref{mu}, \eqref{gamma}, and \eqref{c} corresponding to  the structure maps are $|G|$th roots of unity in $k$.
\end{corollary}
\begin{proof}
By Lemma~\ref{rootof1lemma},   we can assume that values of $\alpha$
are $|G|$th roots of $1$. 
The result follows since
the values of $\mu$ and $\gamma$ in  \eqref{mu}, \eqref{gamma} are products of values of $\alpha$
and its inverses. 
\end{proof}

It was shown in \cite{K, M} (see also \cite[Appendix 5]{T} and \cite[Section 4.4.3]{DGNO})
that the {\em equivariantization}  construction $\C \mapsto \C^G$
gives rise to a  $2$-equivalence between the $2$-category
of braided $G$-crossed fusion categories and the $2$-category 
of braided fusion categories containing $\Rep(G)$. 

The inverse to the equivariantization construction is called {\em de-equivariantization}. It proceeds as follows.
Let $\B$ be a braided fusion category containing  a Tannakian fusion subcategory $\E=\Rep(G)$. The algebra
$\Fun(G)$ of $k$-valued functions on $G$ is a commutative algebra in $\E$ (and, hence, in $\B$). The category 
$\B_G$  of $\Fun(G)$-modules in $\B$ has a canonical structure   of a braided $G$-crossed fusion category.

One has canonical equivalences $(\C^G)_G \cong \C$ of braided $G$-crossed fusion categories and 
$(\B_G)^G \cong \B$ of braided fusion categories. 

Given a braided $G$-crossed fusion category $\C$ with the $G$-crossed braiding $c$  
the braiding  $\tilde{c}$ on $\C^G$ is defined as follows. 
Let $X$ and $Y$  be
objects in $\C^G$. Let $\{v_g^Y: g(Y) \to Y\}_{g\in G}$ denote the $G$-equivariant structure on $Y$
and let  $X =\oplus_{g\in G}\,X_g$ be the decomposition
of $X$ with respect to the grading of $\C$. Set
\begin{equation}
\label{braiding from equiv}
\tilde c_{X,Y}: X\ot Y =\bigoplus_{g\in G}\, X_g\ot Y \xrightarrow{\displaystyle \oplus\,c_{X_g,Y}}
\bigoplus_{g\in G}\, g(Y) \ot X_g \xrightarrow{\displaystyle \oplus\, v_g^Y \ot \id_{X_g}}
\bigoplus_{g\in G}\, Y \ot X_g = Y\ot X.
\end{equation}

Let  $F_\C: \C^G \to \C$ be the tensor functor forgetting the $G$-equivariant structure.  Its right adjoint
$I_\C: \C \to \C^G$  is the induction that can be explicitly described as follows: 
\begin{equation}
\label{ICX}
I_\C(X)=\bigoplus_{g\in G} \, g(X) 
\end{equation}
for any $X\in \C$ with the $G$-equivariant structure given by
\[
v_h^{I_\C(X)}:  h(I_\C(X)) = \bigoplus_{g\in G} \, h\circ g(X)  \xrightarrow{\displaystyle \oplus \gamma_{h,g}(X)}
\bigoplus_{g\in G} \, hg(X) = I_\C(X),\qquad h\in G.
\] 

For  $X\in \C_x,\, x\in G,$ and  $Y\in \C$ the braiding between 
the induced objects $I_\C(X)$ and $I_\C(Y)$  is given by
\begin{multline}
\label{ICXICY}
\tilde{c}_{I_\C(X), I_\C(Y) }:  I_\C(X) \ot I_\C(Y) = \bigoplus_{g,h\in G}\, g(X) \ot h(Y)  \xrightarrow{\displaystyle \oplus_{} c_{g(X),h(Y)}} 
 \bigoplus_{g,h\in G}\,  gxg^{-1}\circ h(Y) \ot g(X) \\
  \xrightarrow{ \displaystyle \oplus_{g}\, v_{gxg^{-1}}^{I_\C(Y)}\ot \text{id}_{g(X)}}   \bigoplus_{g,h\in G}\,  h(Y) \ot g(X)  = I_\C(Y) \ot I_\C(X).
\end{multline}


Let $\C$ be a braided $G$-crossed fusion category. The notion of the {\em reverse} category of $\C$
was considered in \cite{S}. This braided $G$-crossed category
$\C^\rev$ is defined as follows. As an Abelian category, $\C^\rev=\C$ with the same action of $G$. For $X\in\C_g$, $Y\in\C$, 
the tensor product in $\C^\rev$ is $X\ot^\rev Y= X\ot g^{-1}(Y)$ with obvious associativity and unit constraints. 
The $G$-grading on $\C^\rev$ is given by $(\C^\rev)_g=\C_{g^{-1}}$. The $G$-crossed braiding is given by
\begin{equation*}
c^\rev_{X,Y}=c^{-1}_{g^{-1}(Y),g^{-1}h^{-1}g(X)}: X\ot^\rev Y = X \ot g^{-1}(Y) 
\to g^{-1}(Y) \ot g^{-1}h^{-1}g(X) = g^{-1}(Y) \ot^\rev X
\end{equation*}
for all $X\in\C_g$, $Y\in\C_h$.    

In the special case when $G$ is trivial the above notion coincides with that of the reverse of a braided fusion category. 

\begin{proposition}
\label{Crev}
There is a canonical braided equivalence
 $(\C^\rev)^{G}\cong (\C^G)^\rev$ or, equivalently, a braided $G$-crossed equivalence $\C^\rev\cong((\C^G)^\rev)_G$. 
\end{proposition}
\begin{proof}
As  Abelian $k$-linear categories  $(\C^\rev)^G$ and $(\C^G)^\rev$ are equal to $\C^G$. 
We define a tensor  functor $F: (\C^\rev)^G \to (\C^G)^\rev$  as the identity functor equipped
with the tensor structure
\begin{equation}
\label{tensor structure}
X\ot^\rev Y = \bigoplus_{g\in G}\, X_g \ot g^{-1}(Y)  \xrightarrow{\displaystyle \oplus_g \text{id}_{X_{g}}\ot v_{g^{-1}}^Y}  
\bigoplus_{g\in G}\, X_g \ot Y = X\ot Y
\end{equation}
for all $X,\,Y \in \C^G$, where $X=\oplus_{g\in G}\,X_g$ with $X_g\in \C_g$.

Using the definition of the tensor products of $G$-equivariant objects and naturality of the associativity
and braiding constraints one can check directly that $F$ is a braided tensor equivalence.
\end{proof}

\begin{proposition}
\label{Natale's}
Let $\A$ be a weakly group-theoretical fusion category.
There is a braided $G$-crossed fusion category 
\[
\C=\bigoplus_{g\in G}\, \C_g
\] 
such that the trivial component $\C_e$ is pointed and $\Z(\A) \cong \C^G$. 
\end{proposition}
\begin{proof}
Let $\E=\Rep(G)$ be a maximal Tannakian subcategory of $\Z(\A)$.  The corresponding de-equivariantization $\C=\Z(\A)_G$
is a braided $G$-crossed fusion category and $\Z(\A)\cong \C^G$. The trivial component $\C_e$ is called the {\em core} of $\Z(\A)$ and its braided equivalence class
is  independent of the choice of $\E$ \cite{DGNO}. We claim that $\C_e$ is pointed.

Note that $\Z(\A)$ is weakly group-theoretical by \cite{ENO3}.
It was shown in \cite{Na} that the core of a weakly group-theoretical braided fusion category is either pointed or is the Deligne product of a pointed braided fusion
category and an Ising braided fusion category. Thus, $\C_e$  must have one of these forms.  Let $\xi(\M)\in k^\times$  denote the multiplicative central
charge of a modular category $\M$ \cite[Section 6.2]{DGNO}.  Note that $\Z(\A)$ is non-degenerate and weakly integral and, hence, is modular 
(with respect to the canonical spherical  structure on the integral category $\Z(\A)$ \cite{ENO1}). Since the central charge is invariant 
under taking the core,  we have $\xi(\C_e) =\xi(\Z(\A))=1$. This implies that an equivalence  $\C_e\cong \mathcal{P}\bt \mathcal{I}$,
where  $\mathcal{P}$ is pointed and  $\mathcal{I}$ is an Ising category,  is impossible. Indeed, $\xi(\mathcal{P})$ is an $8$th root of $1$ while
$\xi(\mathcal{I})$ is a primitive $16$th root of $1$, so that $\xi(\C_e) = \xi(\mathcal{P}) \xi(\mathcal{I}) \neq 1$. Therefore, $\C_e$ is pointed.
\end{proof}


\section{Determinants in graded fusion categories}
\label{sect det}

Let $\C$ be an integral  fusion category.  
For any object $X$ in $\C$ let $d(X)$ denote the Frobenius-Perron dimension of $X$. 
Given an automorphism $\phi:X \to X$ its {\em determinant} is
\[
\det(\phi) = \prod_{Z\in \O(C)}\,  \det (\phi|_{\Hom_\C(Z,\, X)})^{d(Z)} \in k^\times.
\]

\begin{remark}
Determinants can be defined for morphisms in an arbitrary (i.e., not necessarily integral) fusion category. In general, they take values
in $\mathbb{A}\ot_\mathbb{Z} k^\times$, where $\mathbb{A}$ is the ring of algebraic integers in $\mathbb{R}$.
\end{remark}

Determinants have the following familiar properties \cite[Proposition 2.1]{E}.  

\begin{proposition}
\label{pasha's}
For all objects $X,\,Y$ in $\C$,  automorphisms $\phi,\psi: X \to X$,  $\zeta:Y\to Y$, and  
$\lambda \in k^\times$ we have
\begin{enumerate}
\item[(i)] $\det(\phi\circ \psi)=\det(\phi) \det(\psi)$,
\item[(ii)] $\det(\phi\oplus \zeta) =\det(\phi)\det(\zeta)$,
\item[(iii)] $\det(\lambda \id_X) =\lambda^{d(X)}$, 
\item[(iv)] $\det( \phi\ot \id_Y) =\det(\id_Y \ot \phi) = \det(\phi)^{d(Y)}$.
\end{enumerate}
\end{proposition}


Let $G$ be a finite group and let 
\[
\C=\bigoplus_{g\in G}\, \C_g
\] 
be an integral $G$-graded fusion category. Let $D=\FPdim(\C_e)$. 

Let $R_g =\displaystyle\oplus_{X\in \O(\C_g)}\, d(X)\,X$ be the regular object in $\C_g$. 
We have  $d(R_g)=D$ and
\begin{equation}
\label{RfRg}
{R}_f \ot {R}_g =D \,{R}_{fg}.
\end{equation}

We may and will assume that the abelian category $\C_{reg}$ generated by $R_g,\,g\in G,$ is {\em skeletal}, i.e., that isomorphic
objects of $\C_{reg}$ are equal. 

Let $g_1,\dots, g_n,\,h_1,\dots, h_n\in G$ be  such that $g_1\cdots g_n = h_1\cdots h_n$
and $d_{g_1}\cdots d_{g_n}= d_{h_1}\cdots d_{h_n}$. Any isomorphism
\[
\phi_{R_{g_1},\dots, R_{g_n}}: R_{g_1}\ot \cdots \ot R_{g_n} \to R_{h_1}\ot \cdots \ot R_{h_n}
\]
is identified with an automorphism of $ D^{n-1} R_{g_1\cdots g_n}$

Let $\alpha_{X,Y,Z}: (X\ot Y)\ot Z \to X \ot (Y \ot Z)$ denote the associativity constraint in $\C$. 

\begin{proposition}
\label{alpha C invariant}
The function 
\begin{equation}
\label{alphaCdef}
{\alpha}^\C(f,g,h)= \det({\alpha}_{{R}_f, {R}_g,{R}_h})^D
\end{equation}
is a $3$-cocycle on $G$ with values in $k^\times$.
Its  class  in $H^3(G,\, k^\times)$ is an invariant of the $G$-graded fusion category $\C$.
That is, if $\C$ is a 
$G$-graded fusion category equivalent to $\C$ by a grading preserving tensor equivalence then ${\alpha}^\C$
and ${\alpha}^{\C'}$ are cohomologous.
\end{proposition}
\begin{proof}
The $3$-cocycle condition for $\alpha^\C$ follows from taking the determinants of both sides of the pentagon equation
\[
\alpha_{R_f,R_g, R_h\ot R_i} \circ \alpha_{R_f \ot R_g, R_h,  R_i} =
(\id_{R_f} \ot \alpha_{R_g, R_h, R_i}) \circ \alpha_{R_f,R_g \ot R_h, R_i}  \circ (\alpha_{R_f,R_g, R_h} \ot  \id_{R_i}),
\]
where $f,g,h,i\in G$ and using \eqref{RfRg}. 

Let $F:\C\rightarrow \C'$ be a grading preserving tensor equivalence with a tensor functor structure 
$\varphi_{X,Y}:F(X)\ot F(Y)\xrightarrow{\sim} F(X\ot Y)$. Let $R'_g=F(R_g),\, g\in G,$ denote the homogeneous regular objects of $\C'$.
From the  definition of a tensor functor we have
\[
\varphi_{R_f, R_g\ot R_h} \circ (\id_{R_f} \ot \varphi_{R_g, R_h} ) \circ \alpha_{R'_f,R'_g,R'_h} 
=
F(\alpha_{R_f,R_g,R_h}) \circ  \varphi_{R_f \ot R_g, R_h} \circ ( \varphi_{R_f,R_g} \ot \id_{R_h}),
\]
for all $f,g,h\in G$.  Taking determinants of both sides  of this equation   we get
\[
\frac{{\alpha}^\C(f,g,h)}{{\alpha}^{\C'}(f,g,h)}
=
\frac{\det({\alpha}_{{R}_f,{R}_g,{R}_h})^D}{\det({\alpha}_{{R}'_f,{R}'_g,{R}'_h})^D}
=
\frac{\det({\varphi}_{{R}_f,{R}_{gh}})^{D^2}\det({\varphi}_{{R}_g, R_h})^{D^2}}{\det({\varphi}_{{R}_{fg},{R}_h})^{D^2}
\det(\varphi_{{R}_f,{R}_g})^{D^2}},
\]
i.e., ${\alpha}^\C$ and ${\alpha}^{\C'}$ are cohomologous.
\end{proof}

\begin{example}
Let $G$ be a finite group, $\alpha \in Z^3(G,\, k^\times)$ be a $3$-cocycle, and $N \subset G$ be a normal subgroup.
We can view $\C=\Vec_G^\alpha$ as a $G/N$-graded fusion category. The corresponding $3$-cocycle 
$\alpha^\C\in Z^3(G/N,\, k^\times)$ is given by 
\[
\alpha^\C(fN, gN, hN) = \displaystyle\prod_{x\in fN, y\in gN, z\in hN}\, \alpha(x,\, y,\, z)^{|N|}.
\]
This construction gives rise  to a group homomorphism $H^3(G,\, k^\times) \to H^3(G/N,\, k^\times)$.
\end{example} 

\begin{remark}
Given a $G$-graded fusion category $\C=\oplus_{g\in G}\,\C_g$ and a $3$-cocycle $\omega \in Z^3(G,\, k^\times)$  one 
constructs a new  $G$-graded fusion category $\C(\omega)$ by multiplying the associativity constraint on homogeneous
objects by values of $\omega$ \cite{ENO2}.  From our definition,  $\alpha^{\C(\omega)} =  \alpha^\C \omega^{D^4}$. 
\end{remark}


Let $\C=\displaystyle \oplus_{g\in G}\, \C_g$ be an integral  braided $G$-crossed fusion category 
with structure isomorphisms $\mu$, $\gamma$, 
and $c$, as defined in ~\eqref{mu and gamma} and ~\eqref{crossed braiding}. 
Consider the following functions: 
\begin{eqnarray}
\label{mubar}\mu^\C_g(x,y) &=& \det({\mu}_g(R_x,R_y))^{D^{2}},\\
\label{gammabar}\gamma^\C_{g,h}(x) &=& \det({\gamma}_{g,h}(R_x))^{D^{3}},\\
\label{cbar}{c}^\C_{g,x} &=& \det({c}_{R_g,R_x})^{D^{2}}.
\end{eqnarray}
It follows from axioms  \eqref{G-hex0} - \eqref{G-hex2}
that they define a braided $G$-crossed category structure on $\Vec_G^{{\alpha}^\C}$.

\begin{remark}
\label{Dets are roots of 1}
By Lemma~\ref{rootof1lemma}, there is $r:G^2\to k^\times$ such that the values of $\alpha^\C d(r)$ are $|G|$th roots of $1$. Choose
a function $t:G^2\to k^\times$ such that $t^D=r$. We can replace $\C$ by an equivalent braided $G$-crossed fusion category $\C'$
with the associativity constraint $\alpha'_{X,Y,Z} = d(t)(f,g,h)\, \alpha_{X,Y,Z}$ for $X\in \C_f,\, Y\in \C_g,\, Z\in \C_h,\, f,g,h\in G,$
so that $\alpha^{\C'}=\alpha^\C d(r)$. Thus, we may assume that the values of $\alpha^{\C}$ are $|G|$th roots of $1$.
Similarly, by Proposition~\ref{VecGw Gcrossequiv} and Corollary~\ref{VecGw Gcrossunity} we may assume
that the values of functions \eqref{mubar} - \eqref{cbar} are $|G|$th roots of~$1$.
\end{remark}


For $x,y,g,h\in G$ consider the composition
\begin{equation}
\label{sigmaghxy}
\sigma^{g,h}_{R_x,R_y} : g(R_x)\ot h(R_y) \xrightarrow{c_{ R_{gxg^{-1}}, R_{hyh^{-1}}}} gxg^{-1}\circ h(R_y) \ot g(R_x)
\xrightarrow{\gamma_{gxg^{-1}, h}(R_y) \ot \id} gxg^{-1}h(R_y) \ot g(R_x).
\end{equation}
viewed as an automorphism  of $D\,R_{gxg^{-1}hyh^{-1}}$. These compositions 
are the components of the braiding on the induced object of $\C^G$, see \eqref{ICXICY}.


\begin{corollary}
\label{blocks are roots of 1}
$\C$ is equivalent to a braided $G$-crossed fusion category in which 
\[
\det \left({\sigma}^{g,h}_{R_x,R_y} \right)^{|G|D^2} =1
\]
for all $x,y,g,h\in G$.
\end{corollary}


\section{Braid groups and their representations from braided fusion categories}
\label{sect braid}


Let $n\geq 1$ be an integer. The {\em braid group}  $B_n$ is generated by  elementary braids $\sigma_1,\dots \sigma_{n-1}$
satisfying  the relations $\sigma_i \sigma_{i+1} \sigma_i =  \sigma_{i+1} \sigma_i  \sigma_{i+1}$ for all $i=1,\dots, n-2,$
and $\sigma_i \sigma_j = \sigma_j \sigma_i$ for all $i,j$ such that $|i-j|>1$. The kernel of the surjective homomorphism
from $B_n$ to the symmetric group $S_n$ is  the {\em pure braid group} $P_n$. 

Let $\B$ be a braided fusion category.  For objects $X_1,X_2,\dots,X_n$ of  $\B$ let $X_1 \ot X_2 \ot  \cdots \ot X_n$
denote the tensor product $(\cdots (X_1 \ot X_2)\ot X_3)\ot \cdots \ot X_{n-1})\ot X_n$.  In particular, 
$X^{\ot n}$ will denote the tensor product of $n$ copies of $X$ as above.  

There is a homomorphism
\begin{equation}
\label{pb action}
\rho_X: B_n \to  \Aut_\B( X^{\ot n}),
\end{equation}
defined by 
\begin{multline}
\label{el braid  action}
\rho_X(\sigma_i) : X^{\ot n} \xrightarrow{\alpha} (X^{\ot (i-1)}  \ot (X\ot X)) \ot X^{\ot (n-i-1)}  \\
\xrightarrow{(\id_{X^{\ot (i-1)}}  \ot c_{X, X} )\ot  \id_{X^{\ot (n-i-1)}}} (X^{\ot (i-1)}  \ot (X\ot X)) \ot X^{\ot (n-i-1)} 
\xrightarrow{\alpha^{-1}} X^{\ot n},
\end{multline}
for all $i=1,\dots, n-1$.  Here $\alpha$ denotes the  appropriate composition of associativity isomorphisms of $\B$.
The braid group relations for $\rho_X(\sigma_i)$ follow from the hexagon axioms of the braiding  and the MacLane 
coherence theorem.

The following definition was given by Naidu and Rowell in \cite{NR}.

\begin{definition}
A braided fusion category $\B$ has {\em property (F)} if for every $n\geq 1$ the image of 
\eqref{pb action} is finite for every $X\in \B$.
\end{definition}

\begin{remark}
This property  is equivalent to the image of the homomorphism 
\[
\rho_{X_1,\dots, X_n}: P_n \to \Aut_\B(X_1 \ot X_2 \ot  \cdots \ot X_n)
\] 
defined similarly to \eqref{el braid  action} being finite for all objects $X_1,X_2,\dots,X_n$ in $\B$.
\end{remark}

It was shown in \cite{ERW} that group-theoretical braided fusion categories have property (F).
It was conjectured in \cite{NR} that a braided fusion category $\B$ has property (F) if and only if 
it is weakly integral.


\section{The fiber product of braided $G$-crossed fusion categories}
\label{sect fiber}

The following notion was introduced in \cite{Ni}.

\begin{definition}
 Let $\C^1,\, \C^2$ be fusion categories graded by the same group $G$.  
 The {\em fiber product} of $\C^1$ and $\C^2$ is the  fusion category
 \begin{equation}
 \label{fiber product}
 \C^1 \bt_G \C^2 = \bigoplus_{g\in G}\, \C^1_g \bt \C^2_g. 
 \end{equation}
 \end{definition}
 
Here $\bt$ denotes Deligne's tensor product of Abelian categories. 
Clearly, $\C^1 \bt_G \C^2$ is a fusion subcategory of $\C^1 \bt \C^2$ graded by $G$. 
The trivial component of this grading is $\C^1_e\bt \C^2_e$.

\begin{remark}
Note that neither $\C^1$ nor $\C^2$ are fusion subcategories of $\C^1 \bt_G \C^2$ in general. 
\end{remark}

Suppose that $\C^1,\, \C^2$ are braided $G$-crossed fusion categories. The actions of $G$
on $\C^1$ and $\C^2$ give rise to an action of $G$ on $\C^1 \bt_G \C^2$:
\[
g(X_1\bt X_2) = g(X_1) \bt g(X_2),\qquad X_1\in \C^1_x,\, X_2\in \C^2_x,\, g\in G.  
\]
Similarly, the $G$-crossed braidings $c^1$ and $c^2$ of $\C^1$ and $\C^2$ can be combined to define a $G$-crossed braiding
on $\C^1 \bt_G \C^2$ :
\[
c^{}_{X_1\bt X_2, Y_1\bt Y_2} = c^1_{X_1, Y_1} \bt c^2_{X_2, Y_2}.
\]


%

Recall that a Tannakian subcategory $\E$ of a non-degenerate braided fusion category $\B$ is {\em Lagrangian}
if $\FPdim(\E)^2=\FPdim(\B)$.

\begin{lemma}
\label{Lagr}
Let $\C =\displaystyle\oplus_{g\in G}\, \C_g$ be a non-degenerate braided $G$-crossed fusion category.
Suppose that $\C_e$ contains a $G$-stable Lagrangian subcategory. Then
$\C^G$ is equivalent to the representation category of a twisted group double. In particular,
$\C^G$ is group-theoretical. 
\end{lemma}
\begin{proof}
Let $\E$ be a $G$-stable Lagrangian subcategory of $\C_e$. Then $\E^G$ is 
a Lagrangian subcategory of $\C^G$ and the statement follows from
\cite[Theorem 4.64]{DGNO}.
\end{proof}


\begin{proposition}
\label{when Ce pointed}
Let $\C =\displaystyle\oplus_{g\in G}\, \C_g$ be a non-degenerate braided $G$-crossed fusion category
such that $\C_e$ is pointed. Then $(\C \bt_G \C^\rev)^G$ is group-theoretical. 
\end{proposition}
\begin{proof}
The subcategory of $\C_e \bt \C_e^\rev$ spanned by $\{X\bt X \mid X\in \O(\C_e) \}$ is Lagrangian and $G$-stable.  
So the result follows from Lemma~\ref{Lagr}.
\end{proof}


\section{The finiteness result}
\label{sect finiteness}

Fix $n\geq 1$. Let $G$ be a finite group.
The following action  of $B_n$ on the set $G^{2n}$ was considered in \cite[Section 4.5]{ERW}:
\begin{multline}
\label{permutation pi}
\pi_{n, G}(\sigma_i)  \left( (g_1,h_1), \dots,  (g_n, h_n) \right)  \\
= \left(  (g_1,h_1), \dots,  (g_{i-1}, h_{i-1}), (g_{i}h_{i}g_{i}^{-1}g_{i+1}, h_{i+1}),  (g_{i}, h_{i}), (g_{i+2}, h_{i+2}), \dots,  (g_n, h_n) \right),
\end{multline}
for all $g_j,h_j\in G,\, j=1,\dots, n$. Let $K_{n,G}$ denote the kernel of $\pi_{n, G}$. 

Clearly, $K_{n,G} \subset P_n$ and $[B_n: K_{n,G}]<\infty$.


\begin{proposition}
\label{CGD}
Let $\C=\oplus_{g\in G}\, \C_g$ and  $\D=\oplus_{g\in G}\, \D_g$ be integral braided $G$-crossed  fusion categories.
If $(\C \bt_G \D)^G$ has property (F) then so do $\C^G$ and $\D^G$.
\end{proposition}
\begin{proof}
Consider the regular objects
\begin{eqnarray*}
R&=& \bigoplus_{h\in G}\, R_h,\quad\text{where} \quad R_h=\bigoplus_{X\in \O(\C_h)}\, d(X)\,X, \\  
S&=& \bigoplus_{h\in G}\, S_h,\quad\text{where} \quad S_h=\bigoplus_{X\in \O(\D_h)}\, d(X)\,X.
\end{eqnarray*}
Let us denote $Z_h = R_h\bt S_h$ and let $Z = \oplus_{h\in G}\,Z_h$ in $\C \bt_G \D$. 

Let $I_\C(R)\in \C^G,\, I_{\D}(S)\in \D^G$ be the induced objects \eqref{ICX}. 

Any  object of $\C^G$ (respectively, $\D^G$) is contained in a direct sum of copies of $I_\C(R)$ (respectively, $I_\D(S)$). 
So it suffices to show that the images of homomorphisms
\[
\rho_{I_\C(R)}: B_n \to \Aut_{\C^G}(I_\C(R)^{\ot n}) \quad \text{and} \quad
\rho_{I_{\D}(S)}: B_n \to \Aut_{\D^G}(I_{\D}(S)^{\ot n}) 
\] 
defined in  \eqref{pb action}  are finite.
Since $K_{n,G}$  is a finite index subgroup of $B_n$, it is enough 
to check that  $\rho_{I_\C(R)}({K_{n,G}})$  and  $\rho_{I_\D(S)}({K_{n,G}})$ are finite.

By  the hypothesis, the image of  
\begin{equation}
\label{big rep}
\rho_{I_{\C\bt_G \D }(Z)}: B_n \to \End_{(\C\bt_G\D)^G}(I_{\C\bt_G \D }(Z)^{\ot n})
= \End_{(\C\bt_G\D)^G} \left( \bigoplus_{\substack{g_1,\dots, g_n\in G \\  h_1,\dots, h_n\in G}} \, g_1(Z_{h_1}) \ot \cdots \ot g_n(Z_{h_n}) \right)
\end{equation}
is finite. For each elementary braid $\sigma_i\in B_n,\, i=1,\dots,n-1,$ the automorphism
$\rho_{I_{\C\bt_G \D }(Z)}(\sigma_i)$ maps the summand 
$
g_1(Z_{h_1}) \ot \cdots  \ot g_i(Z_{h_i})  \ot g_{i+1}(Z_{h_{i+1}})   \ot \cdots  \ot  g_n(Z_{h_n})
$
in the direct sum in \eqref{big rep} to
$
g_1(Z_{h_1}) \ot \cdots  \ot  g_{i}h_{i}g_{i}^{-1}g_{i+1}(Z_{h_{i+1}})   \ot  g_i(Z_{h_i})  \cdots  \ot g_n(Z_{h_n}), 
$
for all $g_1,\dots,g_n \in G$ and $h_1,\dots,h_n \in G$. 
Note that these summands are, in general, objects of $\C\bt_G \D$, but not of $(\C\bt_G \D)^G$.  
Comparing with \eqref{permutation pi} we see that 
each of them is stable under $K_{n,G}$.  Let
\begin{equation}
\label{restricted rho}
\rho_{I_{\C\bt_G \D }(Z)}^{(g_1,h_1),\dots, (g_n,h_n)}:
K_{n,G}\to \End_{\C\bt_G \D }(g_1(Z_{h_1}) \ot \cdots \ot g_n(Z_{h_n})). 
\end{equation}
denote the corresponding restrictions  of $\rho_{I_{\C\bt_G \D }(Z)}$.

We have 
\begin{multline}
\End_{\C\bt_G \D }(g_1(Z_{h_1}) \ot \cdots \ot g_n(Z_{h_n})) \\
= \End_\C(g_1(R_{h_1}) \ot \cdots \ot g_n(R_{h_n})) \ot_k  \End_{\D}(g_1(S_{h_1}) \ot \cdots \ot g_n(S_{h_n}))
\end{multline}
and
\begin{equation}
\label{factorization rho}
\rho_{I_{\C\bt_G \D }(Z)}^{(g_1,h_1),\dots, (g_n,h_n)}(\sigma) =
\rho_{I_{\C}(R)}^{(g_1,h_1),\dots, (g_n,h_n)}(\sigma) \ot_k \rho_{I_{\D }(S)}^{(g_1,h_1),\dots, (g_n,h_n)} (\sigma),\qquad \sigma\in K_{n,G}.
\end{equation}
Let 
\[
\Delta(K_{n,G}) =\left\{ \left(\rho_{I_{\C}(R)}^{(g_1,h_1),\dots, (g_n,h_n)}(\sigma),\, \rho_{I_{\D }(S)}^{(g_1,h_1),\dots, (g_n,h_n)} (\sigma) \right) \mid
\sigma \in K_{n,G} \right\} 
\]
be the diagonal subgroup of $\rho_{I_{\C}(R)}^{(g_1,h_1),\dots, (g_n,h_n)}(K_{n,G}) \times \rho_{I_{\D }(S)}^{(g_1,h_1),\dots, (g_n,h_n)} (K_{n,G})$.
We claim that $\Delta(K_{n,G})$ is finite.

Equation  \eqref{factorization rho} defines a surjective group homomorphism
\begin{equation}
\label{hom p}
p: \Delta(K_{n,G}) \to \rho_{I_{\C\bt_G \D }(Z)}^{(g_1,h_1),\dots, (g_n,h_n)}(K_{n,G}).
\end{equation}

Note that $\rho_{I_{\C\bt_G \D }(Z)}^{(g_1,h_1),\dots, (g_n,h_n)}(K_{n,G})$  is finite
since $(\C\bt_G \D)^G$ has property (F). 
The kernel of $p$ is contained in  $L = \{ (\lambda\, \id_{g_1(R_{h_1}) \ot \cdots \ot g_n(R_{h_n})},\, 
\lambda^{-1}\,\id_{g_1(S_{h_1}) \ot \cdots \ot g_n(S_{h_n})} ) \mid \lambda \in k^\times  \} \cap  \Delta(K_{n,G})$.  
Indeed, the Kronecker product of linear operators equals the  identity if and only if the factors are reciprocal scalar
multiples of the identity. Thus, it suffices to prove that $L$ is finite.

For  $\sigma\in K_{n,G}$  let $\sigma =\sigma_{i_1}\sigma_{i_2}\cdots \sigma_{i_m}$ be its presentation as a 
product of elementary braids.  By the definition \eqref{ICXICY} of braiding on the induced object, each 
$\rho_{I_{\C}(R)}(\sigma_{i_k}),\, k=1,\dots,m,$ is a block permutation matrix whose blocks 
are obtained by tensoring maps of the form \eqref{sigmaghxy} with identity and conjugating by the  associativity isomorphisms.
By Corollary~\ref{blocks are roots of 1} the determinant  
of each block of $\rho_{I_{\C}(R)}(\sigma_{i_k})$  is an $N$th root of unity, where $N=|G|D^2$. 

Hence, the determinant of each diagonal block  
${\rho}_{I_{\C}({R})}^{(g_1,h_1),\dots, (g_n,h_n)}(\sigma)$ 
of ${\rho}_{I_{\C}({R})}(\sigma)$
is an $N$th root of unity.
But if  $\rho_{I_{\C}(R)}^{(g_1,h_1),\dots, (g_n,h_n)}(\sigma)=\lambda \,\id_{g_1(R_{h_1}) \ot \cdots \ot g_n(R_{h_n})}$ then,
by Proposition~\ref{pasha's},
\[
\det\left({\rho}_{I_{\C}({R})}^{(g_1,h_1),\dots, (g_n,h_n)}(\sigma)\right) =\lambda^{D^n},
\]
where $D=\FPdim(\C_e)$. Hence, $\lambda^{|G|D^{n+2}} =1$ and $L$ is finite. So $\Delta(K_{n,G})$ is finite.

It follows that  the images $\rho_{I_{\C}(R)}^{(g_1,h_1),\dots, (g_n,h_n)}(K_{n,G})$   and $\rho_{I_{\D}(S)}^{(g_1,h_1),\dots, (g_n,h_n)}(K_{n,G})$ 
are finite for all $g_j,h_j\in G,\, j=1,\dots, n$. Hence, $\rho_{I_{\C}(R)}(K_{n,G})$  and  $\rho_{I_{\D}(S)}(K_{n,G})$ are finite.
\end{proof}

\begin{theorem}
\label{main thm}
Any weakly group-theoretical braided fusion category $\B$ has property (F).
\end{theorem}
\begin{proof}

Let us first prove this theorem in the case when $\B$ is integral.
Note  that $\B$ embeds into its center $\Z(\B)$ which is also integral and weakly group-theoretical by \cite{ENO3}.
So it suffices to prove that $\Z(\B)$ has property (F). By Proposition~\ref{Natale's} we have 
$\Z(\B) =\C^G$, where $\C$ is a braided $G$-crossed fusion category 
with the pointed trivial component. Then $(\C \bt_G \C^\rev)^G$ is group-theoretical by Proposition~\ref{when Ce pointed}.
So the statement follows from  Proposition~\ref{CGD} and \cite{ERW}.

Now suppose that $\B$ is arbitrary and let 
\[
\B= \bigoplus_{a\in G(\B)}\, \B_a
\]
be the grading from Corollary~\ref{dimension grading}. The fiber square $\tilde{\B}= \B\bt_{G(\B)} \B$ is an integral weakly group-theoretical
braided fusion category and so has property (F).  Let $a_1,\dots,a_n \in G(\B)$,  let  $X_i,\, Y_i\in \B_{a_i}$, 
and let $Z_i =X_i \bt Y_i,\, i=1,\dots,n$.  
As in the proof of Proposition~\ref{CGD}, we have
\[
\End_{\tilde{\B}}(Z_1\ot \cdots Z_n) = \End_\B(X_1\ot \cdots \ot X_n) \ot_k \End_\B(Y_1\ot \cdots \ot Y_n),
\]
\[
\rho_{Z_1,\dots, Z_n}(\sigma) = \rho_{X_1,\dots, X_n}(\sigma) \ot_k \rho_{Y_1,\dots, Y_n}(\sigma),\qquad \sigma\in P_n.  
\]
Let 
\[
\Delta(P_n) =\left\{  \left( \rho_{X_1,\dots, X_n}(\sigma) ,\, \rho_{Y_1,\dots, Y_n}(\sigma) \right)\mid \sigma \in P_n \right\}.
\]
To prove that $\Delta(P_n)$ is a finite group we consider a surjective homomorphism
\begin{equation}
\label{rZ=rXrY}
p: \Delta(P_n)  \to \rho_{Z_1,\dots, Z_n}(P_n) : 
( \rho_{X_1,\dots, X_n}(\sigma) ,\, \rho_{Y_1,\dots, Y_n}(\sigma) )  \mapsto \rho_{Z_1,\dots, Z_n}(\sigma).
\end{equation}
Its  kernel  is contained in $L =\{  (\lambda\,\id_{X_1\ot \cdots \ot X_n},\, \lambda^{-1}\,\id_{Y_1\ot \cdots \ot Y_n} ) \mid \lambda\in k^\times \}
\cap \Delta(P_n)$.
Since $\rho_{Z_1,\dots, Z_n}(P_n)$ is finite it suffices to prove that $L$ is finite.

Let $I = \rho_{X_1,\dots, X_n}(P_n)$.  We claim that there are only finitely many scalar automorphisms in $I$. 
Note that the abelian group $I/[I,\,I]$ is finite. Indeed, it is generated by  the images of finitely many conjugates
of pure elementary braids  $\sigma_i^2,\, i=1,\dots n$ \cite{KT}.  These have finite order by \cite{E}.  
Let $N$ be the exponent of $I/[I,\,I]$. If $s=\lambda\, \id_{X_1\ot \cdots \ot X_n} \in I$ then $s^N \in [I,\,I]$
Hence, $\det(s^N)=1 $, i.e., $\lambda^{Nd(X_1)\cdots d(X_n)} =1$. 
So $L$ is finite and the theorem is proved.
\end{proof}


\section{Example: Braid group representations from the centers of Tambara-Yamagami fusion categories}
\label{sect TY}

Recall that a {\em Tambara-Yamagami} fusion category \cite{TY} is a $\mathbb{Z}/2\mathbb{Z}$-graded category whose 
trivial component is pointed and the non-trivial component has rank $1$. Such categories are classified up to a
tensor equivalence by triples $(A,\, \chi,\tau)$, where $A$ is a finite Abelian group, $\chi:A\times A \to k^\times$ is a
non-degenerate symmetric bicharacter, and $\tau$ is a square root of $|A|^{-1}$ in $k$.  Let $\TY(A,\,\chi,\,\tau)$
denote the corresponding fusion category. We denote its invertible objects by $a\in A$ and the non-invertible
simple object by $m$. 

The category $\TY(A,\, \chi,\tau)$ is weakly group-theoretical, but not group-theoretical or integral in general.  Its  center
$\B=\Z(\TY(A,\chi,\tau))$ was studied in detail in \cite{GNN}. The category $\B$ is also $\mathbb{Z}/2\mathbb{Z}$-graded 
with a group-theoretical trivial component. The non-trivial component of $\B$ is spanned by simple objects of 
the Frobenius-Perron dimension $\sqrt{|A|}$. These  are equal to $m$ as objects of $\TY(A,\,\chi,\,\tau)$ and 
are parametrized by pairs  $(q,\,\Delta)$ where $q:A\rightarrow k^{\times}$ 
is a  function satisfying $q(a+b)\chi(a,b)=q(a)q(b)$ for all $a,b\in A$ and $\Delta$ is a scalar such that 
$\Delta^{2}=\tau\sum_{a\in A}q(a)^{-1}$. 

For simplicity, assume that $q$ is a quadratic form, i.e., $q(a)= q(-a)$ for all $a\in A$, and fix $\Delta$.  
Let $Z$ denote the corresponding object of $\B$.  

Let us describe the braid group representations $\rho_Z$ \eqref{pb action} coming from $Z$ (they have finite images
by Theorem~\ref{main thm}.

We have
\begin{equation}
Z^{\ot n} = 
\begin{cases}
\left(\bigoplus_{a\in A}\,m \right)^{\ot \frac{n-1}{2}} & \text{ if $n$ is odd}, \\
\left(\bigoplus_{a\in A}\, a \right)^{\ot \frac{n}{2}} & \text{ if $n$ is even}, 
\end{cases}
\end{equation}
as an object of  $\TY(A,\,\chi,\,\tau)$, 
so that $\End_\B(Z^{\ot n})$ is identified with the vector space $V_n$ 
of dimension $|A|^{(n-1)/2}$  (respectively,  $|A|^{n/2}$) whose basis 
is the set of   $\frac{n-1}{2}$-tuples (respectively, $\frac{n}{2}$-tuples) of elements of $A$ when $n$ is odd
(respectively, even). Note that $V_{2m}\cong  V_{2m-1}$ for all $m\geq 1$. 

For any $n$ let $\tilde{\sigma}_{i}= \rho_Z(\sigma_i),\, i=1,\dots, n-1$.

\begin{proposition}
\label{TY rep}
Let $Z$ be a simple object of $\B$ as described above. The images of the elementary  braids of 
$B_{2m+1}$ in $\Aut_{\B}(Z^{\ot (2m+1)})$  are
\begin{eqnarray*}
\widetilde{\sigma}_{2i-1}&:& (a_{1},\dots,a_{n})\mapsto \Delta  q(a_{i-1}-a_{i})^{-1}(a_{1},\dots,a_{n}),\qquad 1\leq i\leq n, \\
\widetilde{\sigma}_{2i} &:& (a_{1},\dots,a_{n})\mapsto  \tau\Delta^{3} \sum_{b\in A} q(a_{i}-b) (a_{1},\dots,a_{i-1},b,a_{i+1},\dots,a_{n}),\quad 1\leq i\leq n,
\end{eqnarray*}
where we set $a_{0}=0$.
\end{proposition}


\begin{proof}
These are direct computations using explicit formulas  for the  braiding isomorphisms of $\B$ from \cite{GNN}. 
\end{proof}

\begin{remark}
The image of the braid group representation  described in Proposition~\ref{TY rep} is finite by  Theorem ~\ref{main thm}.
Our formulas are equivalent to those in \cite{BW} for the images of mapping class groups arising from abelian anyon models.
\end{remark}

\bibliographystyle{ams-alpha}

\begin{thebibliography}{A} 
 
\bibitem [BW]{BW} W.~Bloomquist, Z.~Wang, 
\textit{On topological quantum computing with mapping class group representations}, 
arXiv:1805.04622v3 [math.QA] (2018).

\bibitem [DGNO]{DGNO}  V.~Drinfeld, S.~Gelaki, D.~Nikshych, and V.~Ostrik.
\textit{On braided fusion categories I}, 
Selecta Mathematica,  \textbf{16}  (2010), no.\ 1,  1~{-}~119.

\bibitem [DMNO]{DMNO}  A.~Davydov, M.~M\"uger, D.~Nikshych, and V.~Ostrik,
{\em The Witt group of non-degenerate braided fusion categories},
Journal f\"ur die reine und angewandte Mathematik,
\textbf{677} (2013), 135--177.

\bibitem[E]{E} P.~Etingof,
\textit{On Vafa's theorem for tensor categories},
Mathematical Research Letters \textbf{9}, 651-657 (2002). 

\bibitem[ERW]{ERW} P.~Etingof, E.~Rowell, and S.~Witherspoon,
\textit{Braid group representations from twisted quantum doubles of finite groups},
Pacific J.\ Math.\ \textbf{234} (2008) , no.\ 1, 33-42.

\bibitem[EGNO]{EGNO} P.~Etingof, S.~Gelaki, D.~Nikshych, and V.~Ostrik, 
{\em  Tensor categories}, 
Mathematical Surveys and Monographs,
\textbf{205}, American Mathematical Society (2015).

\bibitem[ENO1]{ENO1}   P.~Etingof, D.~Nikshych, and V.~Ostrik.
{\em  On fusion categories},  Annals of Mathematics \textbf{162} (2005), 581--642.

\bibitem[ENO2]{ENO2}   P.~Etingof, D.~Nikshych, and V.~Ostrik.
{\em  Fusion categories and homotopy theory}, 
Quantum Topology,    \textbf{1} (2010), no.\ 3,  209-273.

\bibitem[ENO3]{ENO3}   P.~Etingof, D.~Nikshych, and V.~Ostrik.
{\em  Weakly group-theoretical and solvable fusion categories}, 
Adv.\ Math, \textbf{226} (2011),  no.\,1, 176--205

\bibitem[GNN]{GNN} S.~Gelaki, D.~Naidu, and D.~Nikshych,
{\em  Centers of graded fusion categories}, 
Algebra and Number Theory, \textbf{3} (2009), no.\ 8, 959--990. 

\bibitem[GN]{GN} S.~Gelaki and D.~Nikshych,
{\em Nilpotent fusion categories}, 
\textit{Advances in Mathematics} \textbf{217} (2008), no. 3, 1053--1071.

\bibitem [KT]{KT} C.~Kassel, V.~Turaev, 
\textit{Braid Groups}, Graduate Texts in Mathematics \textbf{247}, Springer (2008).

\bibitem[K]{K} A.~Kirillov Jr.,
{\em Modular Categories and Orbifold Models},
Communications in Mathematical Physics,
\textbf{229}, 309--335 (2002).

\bibitem [M]{M} M.~M\"uger,
{\em Galois extensions of braided tensor categories and
braided crossed G-categories}, J. Algebra {\bf 277} (2004),
no. 1, 256--281.

\bibitem[Na]{Na} S.~Natale, 
{\em The core of a weakly group-theoretical braided fusion category},
Internat.\ J.\ Math.\ \textbf{29} (2018), no. 2, 1850012.

\bibitem[Ni]{Ni} D.~Nikshych,
{\em Classifying braidings on fusion categories}, 
Tensor categories and Hopf algebras,
Contemporary Mathematics \textbf{728}  (2019),  151-163. 


\bibitem[NR]{NR} D.~Naidu, E.~Rowell,
{\em A finiteness property for braided fusion categories},
Algebras and Representation Theory, \textbf{14} (2011), 837-855. 


\bibitem[RW]{RW} E.~Rowell, H.~Wenzl, 
{\em $SO(N)_2$ braid group representations are Gaussian},
Quantum Topology \textbf{8} (2017), no. 1, 1-33.

\bibitem[S]{S} Raza Sohail Sheikh,
{\em Group-crossed extensions of representation categories
in algebraic quantum field theory}, PhD thesis in preparation.
Radboud University, Nijmegen.

\bibitem [TY]{TY} D.~Tambara, S.~Yamagami, 
{\em Tensor categories with fusion rules of self-duality for finite abelian groups}, 
J.\ Algebra \textbf{209} (1998), no. 2, 692-707.

%

\bibitem [T]{T} V.~Turaev,
{\em Homotopy quantum field theory}, with Appendix 5 by Michael M\"uger and Appendices 6 and 7 by Alexis Virelizier, 
EMS Tracts in Mathematics, \textbf{10}, European Mathematical Society (2010).
%

\end{thebibliography}

\end{document}